\documentclass{article}
\usepackage{epsfig}
\usepackage{amsmath}
\usepackage{amsfonts}
\usepackage{amsthm}
\usepackage{algorithm}
\usepackage{algorithmic}
\usepackage{graphicx}
\usepackage{subfigure}

\newcommand {\C}        {{\mathbb{C}}}
\newcommand {\R}        {{\mathbb{R}}}

 \newtheorem{theorem}{Theorem}[section]
 \newtheorem{lemma}[theorem]{Lemma}
 
 \newtheorem{corollary}[theorem]{Corollary}
 \newtheorem{definition}[theorem]{Definition}

\title{ Generalized Eigenvalue Problems with Specified Eigenvalues }
\author{ Daniel~Kressner\thanks{SB MATHICSE ANCHP, EPF Lausanne,
Station 8,	CH-1015 Lausanne, Switzerland
{\tt (daniel.kressner@epfl.ch)}.} \and
 Emre~Mengi\thanks{ Department of Mathematics, Ko\c{c} University,
Rumelifeneri Yolu, 34450 Sar{\i}yer-\.{I}stanbul, Turkey {\tt (emengi@ku.edu.tr)}. }\and
Ivica Naki\'c\thanks{Department of Mathematics, University of Zagreb, 
Bijeni\v{c}ka 30, 10000 Zagreb, Croatia {\tt (nakic@math.hr)}. } \and
Ninoslav Truhar\thanks{ Department of Mathematics, University of Osijek,
Trg Ljudevita Gaja 6, HR-31 000 Osijek, Croatia {\tt (ntruhar@mathos.hr)}. }
}

\begin{document}
\maketitle

\begin{abstract}
We consider the distance from a (square or rectangular) matrix pencil to the 
nearest matrix pencil in 2-norm that has a set of specified 
eigenvalues. We derive a singular value optimization characterization for this 
problem and illustrate its usefulness for two applications. First, the characterization yields a singular value formula for determining the nearest pencil 
whose eigenvalues lie in a specified region in the complex plane. For instance, this enables
the numerical computation of the nearest stable descriptor system in control theory.
Second, the characterization partially solves the problem posed in [Boutry et al. 2005] regarding the 
distance from a general rectangular pencil to the nearest pencil with a complete set of eigenvalues.
The involved singular value optimization problems are solved by means of BFGS and
Lipschitz-based global optimization algorithms. \\ \\


\end{abstract}

\pagestyle{myheadings}
\thispagestyle{plain}
\markboth{D. KRESSNER, E. MENGI, I. NAKIC AND N. TRUHAR}{Generalized Eigenvalue Problems with Specified Eigenvalues}

\section{Introduction}
Consider a matrix pencil $A - \lambda B$ where $A, B \in \C^{n\times m}$ with $n \geq m$.
Then a scalar $\rho \in \C$ is called an \emph{eigenvalue} of the pencil if there 
exists a nonzero vector $v \in \C^n$ such that
\begin{equation}\label{eq:eig_defn}
	 ( A - \rho B ) v = 0.
\end{equation}
The vector $v$ is said to be a \emph{(right) eigenvector} associated with $\rho$
and the pair $(\rho, v)$ is said to be an \emph{eigenpair} of the pencil.

In the square case $m=n$, the eigenvalues are simply given by the roots
of the characteristic polynomial $\det(A-\lambda B)$ and 
there are usually $n$ eigenvalues, counting multiplicities.
The situation is quite the opposite for $n>m$.
Generically, a rectangular pencil $A - \lambda B$ has
no eigenvalues at all. To see this, notice that a necessary condition
for the satisfaction of~(\ref{eq:eig_defn}) is that $ n! / \left( (n-m)! m! \right)$ polynomials, each corresponding to the determinant of a pencil obtained by choosing $m$ 
rows of $A - \lambda B$ out of $n$ rows, must have a common root. 
Also, the generic Kronecker canonical form of a rectangular
matrix pencil only consists of singular blocks (see \cite{DemE95}).
Hence, (\ref{eq:eig_defn}) is an ill-posed problem and requires reformulation
before admitting numerical treatment.

To motivate our reformulation of~(\ref{eq:eig_defn}), we describe a typical situation giving rise to
rectangular matrix pencils. Let $M \in \mathbb C^{n\times n}$
and suppose that the columns of 
$U \in \mathbb C^{n\times m}$ form an orthonormal basis
for a subspace $\mathcal W \subset \mathbb C^n$ 
known to contain approximations to some eigenvectors
of $M$. Then it is quite natural
to consider the $n\times m$ matrix pencil
\begin{equation} \label{eq:rectpencilsubspace}
 A - \lambda B := M U - \lambda U.
\end{equation}
The approximations contained in $\mathcal W$ and the approximate eigenpairs of $A - \lambda B$ are closely connected
to each other.
In one direction, suppose that $(\rho, x)$ with $x \in \mathcal W$
satisfies \begin{equation} \label{eq:squarematrixperturbed}
(M+\Delta M - \rho I) x = 0
\end{equation}
for some (small) perturbation $\Delta M$. Then there is $v \in \mathbb C^{n}$
such that $x = Uv$. Moreover, we have
\begin{equation} \label{eq:rectpencilsubspaceperturbed}
(A+\Delta A - \rho B)v = 0
\end{equation}
with $\Delta A := \Delta M\cdot U$ satisfying
$\|\Delta A\|_2 \le \|\Delta M\|_2$.
In the other direction, the relation~(\ref{eq:rectpencilsubspaceperturbed}) with
an arbitrary $\Delta A$ implies~(\ref{eq:squarematrixperturbed})
with $\Delta M = \Delta A\cdot U^{\ast}$ satisfying $\|\Delta M\|_2 = \|\Delta A\|_2$.
Unless $M$ is normal, the first part of this equivalence between approximate eigenpairs of $M$ and $A - \lambda B$
does not hold when the latter is replaced by the more common
compression $U^{\ast} M U$. This observation has led to the use of rectangular matrix
pencils in, e.g., large-scale pseudospectra computation (see \cite{TohT96}) and Ritz
vector extraction (see \cite{JiaS01}).

This paper is concerned with determining 
the 2-norm distance from the pencil $A - \lambda B$ to the
nearest pencil $(A+\Delta A) - \lambda B$ with a subset of specified eigenvalues. To be
precise, let ${\mathbb S}= \{ \lambda_1, \dots, \lambda_k \}$ be a set
of  distinct complex numbers and let $r$ be a positive integer. Let
$m_j(A + \Delta A, B)$ denote the (possibly zero) algebraic multiplicity\footnote{For a rectangular matrix pencil,
the algebraic multiplicity of $\lambda_j$ is
defined as the sum of the sizes of associated regular Jordan blocks in the Kronecker canonical form, see also Section~\ref{sec:KCF}.
By definition, this number is zero if $\lambda_j$ is actually not an eigenvalue of the pencil.}
of $\lambda_j$ as an eigenvalue  of
$(A+\Delta A) - \lambda B$. Then we consider the distance
\begin{equation} \label{eq:taurs}
 	\tau_r({\mathbb S}) := 
			\inf
				\Big\{
					\| \Delta A \|_2 : \sum_{j=1}^k m_j(A + \Delta A, B) \geq r
				\Big\}.
\end{equation}
We allow $B$ to be rank-deficient. However, we require that ${\rm rank}(B) \geq r$. Otherwise,
if ${\rm rank}(B) < r$,
the pencil $(A+\Delta A) - \lambda B$ has fewer than $r$ finite eigenvalues for
all $\Delta A$ and consequently the distance $\tau_r({\mathbb S})$ is ill-posed.

For $k = r = 1$, it is relatively easy to see that
\[
\tau_1(\{\lambda_1\}) = \sigma_m(A - \lambda_1 B),
\]
where, here and in the following, $\sigma_k$ denotes the $k$th largest singular value of a matrix.
(The particular form of this problem with $k=r=1$, and when $A$ and $B$ are perturbed simultaneously,
is also studied for instance in \cite{Byers93}.) 
One of the main contributions of this paper is a derivation of a similar
singular value optimization characterization for general $k$ and $r$,
which facilitates the computation of $\tau_r({\mathbb S})$. Very little
seems to be known in this direction. Existing results
concern the square matrix case ($m=n$ and $B = I$); see
the works by \cite{Malyshev1999} for $k = 1$ and $r = 2$ as well as \cite{MR2156435}
for $k = 2$ and $r=2$, \cite{MR1991900} for $k = 1$ and $r = 3$, and
\cite{Mengi2009} for $k = 1$ and arbitrary $r$. Some attempts have also
been made by \cite{MR2592917} for arbitrary $k$ and $r$ and for the square
matrix case, and by \cite{MR2444335} for $k = 1$ and $r = 2$ and 
for the square matrix polynomial case. 

Another class of applications arises in (robust) control theory, where a number of tasks 
 require the determination of a (minimal)
perturbation that moves some or all eigenvalues into a certain
region in the complex plane. With the region of interest
denoted by 
$\Omega \subseteq \mathbb{C}$, the results in this paper are an important
step towards rendering the numerical computation of 
the distance
\begin{eqnarray*}
	\tau_r(\Omega)& := &
			\inf
				\big\{
					\| \Delta A\|_2 : (A+\Delta A) - \lambda B \; \text{has $r$ finite eigenvalues in $\Omega$}
				\big\} \\
				&= &
			\inf_{{\mathbb S} \subseteq \Omega} \tau_r({\mathbb S})
\end{eqnarray*}
feasible.
Here and in the following, multiple eigenvalues are counted according to their algebraic multiplicities.
For $r = 1$ and $\Omega$ equal to $\mathbb C^+$ (right-half complex plane),
the quantity $\tau_1(\mathbb C^+)$ amounts to the
distance to instability, also called stability radius. In \cite{VanLoan1985},
a singular value characterization of $\tau_1(\mathbb C^+)$
was provided, forming the basis of a number of algorithms
for computing $\tau_1(\mathbb C^+)$, see, e.g.,~\cite{BoyB90,Bye88}.
In our more general setting, we can also address the converse question:
Given an unstable matrix pencil $A-\lambda B$, determine
the closest stable pencil. Notice that this problem is intrinsically harder than the distance 
to instability. For the distance to instability it suffices to perturb the system so that 
\emph{one of the eigenvalues} is in the undesired region. On the other hand to make an 
unstable system stable one needs to perturb the system so that \emph{all eigenvalues} lie in the 
region of stability.

An important special case, $\Omega = \mathbb C$ leads to
\begin{eqnarray*}
	\tau_r(\C) &:= &
		\inf \{ \| \Delta A \|_2 : (A+\Delta A) - \lambda B \;\; {\rm has} \; r \; {\rm finite} \; {\rm eigenvalues} \; \} \\
			&= &
		\inf_{{\mathbb S} \subseteq \C} \tau_r({\mathbb S}). 
\end{eqnarray*}
For $r = 1$ and particular choices of rectangular $A$ and $B$, the distance $\tau_1(\C)$
corresponds to the distance to uncontrollability for a matrix pair (see \cite{BurLO04,Eis84}).
For general $r$, a variant of this distance was suggested in \cite{Boutry2005} to solve an inverse signal
processing problem approximately. More specifically, this problem
is concerned with the identification of the shape of a region in the complex plane given the moments over the
region. If the region is assumed to be a polygon, then its vertices can be posed as the eigenvalues
of a rectangular pencil $A - \lambda B$, where $A$ and $B$ are not exact due to 
measurement errors, causing the pencil to have no eigenvalues (see \cite{Elad2004} for details). 
Then the authors attempt to locate nearby pencils with a complete set of eigenvalues.
In this work we allow perturbations to $A$ only, but not to $B$. This restriction is only justified
if the absolute value of $\lambda$ does not become too small.
We consider our results and
technique as significant steps towards the complete solution of the problem posed in \cite{Elad2004}.

The outline of this paper is as follows.
In the next section, we review the Kronecker
canonical form for the pencil $A - \lambda B$.
In \S\ref{sec:rank_char_pencils}, we derive a rank characterization 
for the condition $\sum_{j=1}^k m_j(A , B) \geq r$. This is a crucial prerequisite for
deriving the 
singular value characterizations of $\tau_r({\mathbb S})$ in \S\ref{sec:SVD_derivation}.  
We discuss several corollaries of the singular value characterizations for $\tau_r({\mathbb S})$, in particular
for $\tau_r(\Omega)$ and $\tau_r(\C)$, in \S\ref{sec:nearest_rect_pencil}.
The singular value characterizations are
deduced under certain mild multiplicity and linear independence assumptions. Although we expect these assumptions to be satisfied for
examples of practical interest, they may fail to hold as demonstrated by an academic example in \S\ref{sec:qualifications}. Interestingly, the
singular value characterization remains true for this example
despite the fact that our derivation no longer applies. 
Finally, a numerical approach to solving the involved singular value optimization problems is briefly outlined 
in \S\ref{sec:computation} and applied to a number of settings in 
\S\ref{sec:numerical_exp}.
The main point of the developed numerical method and the experiments is to demonstrate that 
the singular value characterizations facilitate the computation of 
$\tau_r({\mathbb S})$, $\tau_r(\Omega)$ and $\tau_r(\C)$. 
We do not claim that the method outlined here is as efficient as it
could be, neither do we claim that it is reliable. 

\section{Kronecker Canonical Form}\label{sec:KCF}

Given a matrix pencil $A-\lambda B  \in \C^{n \times m}$, the Kronecker
canonical form (KCF), see~\cite{Gantmacher1959}, states the existence of invertible matrices $P \in \C^{n\times n}$ and $Q \in \C^{m\times m}$ such that the
transformed pencil $P(A-\lambda B)Q$ is block diagonal with each diagonal block taking the form
\[
	J_p(\alpha) - \lambda I_p \quad \text{or} \quad I_p - \lambda J_p(0) \quad \text{or}\quad F_p - \lambda G_p \quad \text{or}\quad F_p^T - \lambda G_p^T,
\]
where
\begin{equation} \label{eq:blockdef}
	J_p(\alpha) =  \underbrace{\left[
\begin{array}{cccc}
 \alpha & 1 \\[-5pt]
	& \alpha & \ddots \\[-5pt]
	& & \ddots & 1 \\
	& & & \alpha
\end{array}\right]}_{p\times p},\ F_p = \underbrace{\left[
\begin{array}{cccc}
 1 & 0 \\
 & \ddots & \ddots \\
 & & 1 & 0
\end{array} \right]}_{p \times (p+1)},\ G_p = \underbrace{\left[
\begin{array}{cccc}
 0 & 1 \\
 & \ddots & \ddots \\
 & & 0 & 1
\end{array} \right]}_{p \times (p+1)}
\end{equation}
for some $\alpha \in \C$. \emph{Regular blocks} take the form $J_p(\alpha) - \lambda I_p$ or $I_p - \lambda J_p(0)$, with $p\ge 1$, corresponding to finite or infinite eigenvalues, respectively. The blocks $F_p - \lambda G_p$ and $F_p^T - \lambda G_p^T$ are called \emph{right and left singular blocks}, respectively,
with $p\ge 0$ corresponding to a so called Kronecker index.

In large parts of this paper, indeed until the main singular value optimization characterization, 
we will assume that $A-\lambda B$ has no right singular blocks $F_p - \lambda G_p$.
Eventually, we will remove this assumption
by treating the occurence of such blocks separately in Section~\ref{sec:mainresult}.

\section{Rank Characterization for Pencils with Specified Eigenvalues}\label{sec:rank_char_pencils}


In this section we derive a rank characterization for the satisfaction of the condition
\begin{equation}\label{eq:alg_mult_ineq}
	\sum_{j=1}^k m_j(A,B) \geq r,
\end{equation}
where $m_j(A,B)$ denotes the algebraic multiplicity of the eigenvalue $\lambda_j$.
The following classical result \cite[Theorem 1, p. 219]{Gantmacher1959} concerning the dimension 
of the solution space for a Sylvester equation will play a central role. 
\begin{theorem}\label{thm:Sylvester_matrix}
	Let $F \in \C^{m\times m}$ and $G \in \C^{r\times r}$. Then the dimension of the solution space for the
	Sylvester equation
	\[
		FX - XG = 0
	\]
	only depends on the Jordan canonical forms of the matrices $F$ and $G$.
	Specifically, suppose that $\mu_1, \dots, \mu_\ell$ are the common eigenvalues of $F$
	and $G$. Let $c_{j,1}, \dots, c_{j, \ell_j}$ and $p_{j,1}, \dots, p_{j,\tilde{\ell}_j}$ denote the sizes 
	of the Jordan blocks of $F$ and $G$ associated with the eigenvalue $\mu_j$, respectively.
	Then
	\[
		{\rm dim} \{ X \in \C^{m\times r} : FX - XG = 0 \}		
						=
			\sum_{j = 1}^\ell \sum_{i = 1}^{\ell_j} \sum_{q = 1}^{\tilde{\ell}_j}
					\min ( c_{j,i}, p_{j,q}	). 
	\]
\end{theorem}

For our purposes, we need to extend the result of Theorem~\ref{thm:Sylvester_matrix}
to a generalized Sylvester equation of the form
\begin{equation}\label{eq:gen_Sylvester_pr}
	AX - BXC = 0,
\end{equation}
where $C$ is a matrix with the desired set of eigenvalues ${\mathbb S}$ and with correct algebraic multiplicities.
For this type of generalized Sylvester equation, the extension is straightforward.\footnote{\cite{Kosir96} provides an extension of Theorem~\ref{thm:Sylvester_matrix} to
 a more general setting.} To see this, let us partition the KCF
\begin{equation} \label{eq:kcfpart}
 		P(A - \lambda B) Q = {\rm diag}
					\left(
						A_F - \lambda I, I - \lambda A_I, A_S - \lambda B_S
					\right),
\end{equation}
such that
\begin{itemize}
	\item $A_F - \lambda I$ contains all regular blocks corresponding to finite eigenvalues;
\item $I - \lambda A_I$ contains all regular blocks corresponding to infinite eigenvalues;
	\item $A_S - \lambda B_S$ contains all left singular blocks of the form $F_p^T - \lambda G_p^T$.
\end{itemize}
As explained in Section~\ref{sec:KCF}, we exclude the occurence of right singular blocks for the moment.
Note that the finite eigenvalues of $A - \lambda B$ are 
equal to the eigenvalues of $A_F$ with the same algebraic and geometric multiplicities.

Using~(\ref{eq:kcfpart}), $X$ is a solution of the generalized Sylvester equation (\ref{eq:gen_Sylvester_pr}) if and only if
\[
	 (PAQ) (Q^{-1} X) - (PB Q) (Q^{-1} X) C = 0 \;\;\;
	 		        \Longleftrightarrow \;\;\;
			       		{\rm diag}
						\left(
							A_F, I, A_S
						\right)  Y
									-
					{\rm diag}
						\left(
							I, A_I, B_S
						\right) Y C = 0
\]
where $Y = Q^{-1}X$. Consequently, the dimension of the solution space for~(\ref{eq:gen_Sylvester_pr}) is the sum of the solution space 
dimensions of the equations
\[
	A_F Y_1   -  Y_1 C   = 0 \;\;\;\; {\rm and} \;\;\;\; 
Y_2   - A_I Y_2 C   = 0
\;\;\;\; {\rm and} \;\;\;\; A_S Y_3 - B_S Y_3 C = 0. 
\]
Results by~\cite{DemE95} show that the last two equations only admit the trivial solutions $Y_2 = 0$ and $Y_3 = 0$.
To summarize: the solution spaces of the generalized Sylvester equation (\ref{eq:gen_Sylvester_pr}) 
and the (standard) Sylvester equation
\[
		A_F X - XC = 0
\]
have the same dimension. Applying Theorem \ref{thm:Sylvester_matrix} we therefore obtain the following result.
\begin{theorem}\label{thm:Sylvester}
	Let $A, B \in \C^{n\times m}$ with $n \geq m$ be such that the KCF of $A-\lambda B$ does not contain right singular blocks. Then the dimension of the solution space for the generalized Sylvester equation
	\[
		AX - BXC = 0
	\]
	only depends on the Kronecker canonical form of $A - \lambda B$ and the Jordan canonical form of  $C \in \C^{r\times r}$.
	Specifically suppose that $\mu_1, \dots, \mu_\ell$ are the common eigenvalues of $A - \lambda B$
	and $C$. Let $c_{j,1}, \dots, c_{j, \ell_j}$ and $p_{j,1}, \dots, p_{j,\tilde{\ell}_j}$ denote the sizes 
	of the Jordan blocks of $A - \lambda B$ and $C$ associated with the eigenvalue $\mu_j$, respectively.
	Then
	\[
		{\rm dim} \{ X \in \C^{m\times r} : AX - BXC = 0 \}		
						=
			\sum_{j = 1}^\ell \sum_{i = 1}^{\ell_j} \sum_{q = 1}^{\tilde{\ell}_j}
					\min ( c_{j,i}, p_{j,q}	). 
	\]
\end{theorem}

We now apply the result of Theorem~\ref{thm:Sylvester} to the generalized Sylvester equation 
\begin{equation}\label{eq:gen_Sylvester}
	AX - BXC(\mu,\Gamma) = 0,
\end{equation}
where $C(\mu,\Gamma)$ takes the form
\begin{equation} \label{eq:cmu}
 	C(\mu,\Gamma)
		=
		\left[
			\begin{array}{cccc}
				\mu_1 & -\gamma_{21} & \dots & -\gamma_{r1} \\
					0 &	\mu_2	     & \ddots & \vdots \\[-0.1cm]
					   &			     & \ddots & -\gamma_{r,r-1}		\\
					 0 &			     &		  & \mu_r \\		
			\end{array}
		\right],
\end{equation}
with
	\[
	  \mu 
			=
		\left[
			\begin{array}{cccc}
				\mu_1 	&	\mu_2	& 	\dots 	&	\mu_r
			\end{array}
		\right]^T \in {\mathbb S}^r, \quad 
	  \Gamma
			=
		\left[
			\begin{array}{cccc}
				\gamma_{21} & \gamma_{31} & \dots & \gamma_{r,r-1}
			\end{array}
		\right]^T \in \C^{r(r-1)/2}.	
	\]
As explained in the introduction, the set ${\mathbb S} = \{ \lambda_1,\dots, \lambda_k \}$ contains
the desired approximate eigenvalues. Suppose that $\lambda_j$ occurs $p_j$ times in $\mu$. Furthermore, as in Theorem \ref{thm:Sylvester},
denote the sizes of the Jordan blocks of $A - \lambda B$ and $C(\mu,\Gamma)$ associated with 
the scalar $\lambda_j$ by $c_{j,1}, \dots, c_{j, \ell_j}$ and $p_{j,1}, \dots, p_{j,\tilde{\ell}_j}$, respectively. Note that
$p_j = \sum_{q=1}^{\tilde{\ell}_j} p_{j,q}$. In fact, for generic values of $\Gamma$ the matrix 
$C(\mu,\Gamma)$ has at most one Jordan block of size $p_j$ associated with 
$\lambda_j$ for $j = 1, \dots, k$, see~\cite{DemE95}. In the following, we denote this set of generic values 
for $\Gamma$ by ${\mathcal G}(\mu)$. By definition, this set depends on $\mu$ but not on $A - \lambda B$.

First, suppose that inequality (\ref{eq:alg_mult_ineq}) holds. If we choose $\mu$ such that  $\sum_{j=1}^k p_j = r$
and $p_j \leq m_j(A,B) = \sum_{i = 1}^{\ell_j} c_{j,i}$, then Theorem \ref{thm:Sylvester} implies that 
the dimension of the solution space for the generalized Sylvester equation (\ref{eq:gen_Sylvester}) is 
\[
		\sum_{j = 1}^k \sum_{i = 1}^{\ell_j} \sum_{q = 1}^{\tilde{\ell}_j}
					\min ( c_{j,i}, p_{j,q}	)
						\geq
		\sum_{j = 1}^k \sum_{i = 1}^{\ell_j} 
					\min (  c_{j,i}, p_j  )
						\geq
		\sum_{j = 1}^k \min ( m_j(A,B), p_j )
						  =
		\sum_{j = 1}^k p_j = r.
\]
In other words, there exists a vector $\mu$ with components from $\mathbb{S}$
such that the dimension of the solution space of the Sylvester equation (\ref{eq:gen_Sylvester})
is at least $r$.

Now, on the contrary, suppose that inequality (\ref{eq:alg_mult_ineq}) does
not hold. Then for generic values $\Gamma \in {\mathcal G}(\mu)$, the solution space 
dimension of~(\ref{eq:gen_Sylvester}) is
\[
	\sum_{j = 1}^k \sum_{i = 1}^{\ell_j}   \min (  c_{j,i}, p_{j}  )					
						\leq
		\sum_{j = 1}^k \sum_{i = 1}^{\ell_j} 
							c_{j,i}
						=
		\sum_{j=1}^k m_j(A,B) < r.			
\]
In other words, no matter how $\mu$ is formed from $\mathbb{S}$,
the dimension is always less than $r$ for
$\Gamma \in {\mathcal G}(\mu)$. This shows the following result.
\begin{theorem}\label{thm:spec_eigvals_Syl}
Let $A, B \in \C^{n\times m}$ with $n \geq m$ be such that the KCF of $A-\lambda B$ does not contain right singular blocks.
Consider a set ${\mathbb S} = \{ \lambda_1, \dots, \lambda_k \}$ of distinct complex scalars, 
and a positive integer $r$. Then the following two statements are equivalent.
\begin{enumerate}
	\item[\bf (1)] $\sum_{j=1}^k m_j(A,B) \geq r$, where $m_j(A,B)$ is the algebraic multiplicity of $\lambda_j$
	as an eigenvalue of $A - \lambda B$.
	\item[\bf (2)] There exists $\mu \in {\mathbb S}^{r}$ such that 
	\[
		{\rm dim} \{ X \in \C^{m\times r} : AX - BXC(\mu,\Gamma) = 0 \} \ge r
	\]
for all $\Gamma \in {\mathcal G}(\mu)$, where $C(\mu,\Gamma)$ is defined as in~(\ref{eq:cmu}).
\end{enumerate}
\end{theorem}

\noindent To obtain a matrix formulation of Theorem~\ref{thm:spec_eigvals_Syl},
we use the Kronecker product $\otimes$ to vectorize the generalized Sylvester 
equation~(\ref{eq:gen_Sylvester}) and obtain
\[
	\left(
		((I \otimes A) - (C^T(\mu,\Gamma) \otimes B)
	\right)
		{\rm vec}(X)   =  {\mathcal L}(\mu,\Gamma,A,B) {\rm vec}(X) = 0,
\] 
with the lower block triangular matrix
\begin{equation}\label{eq:linear_pencil}
			{\mathcal L}(\mu,\Gamma,A,B) 
			:= 
		\left[
			\begin{array}{ccccc}
				A-\mu_1 B & & & &  \\
				 \gamma_{21} B & A-\mu_2 B & & &  \\
				     \vdots & \ddots & \ddots & & \\
				      \vdots & & \ddots & A-\mu_{r-1} B &  \\
				  \gamma_{r 1} B & \gamma_{r 2} B & \cdots & \gamma_{r, r -1} B & A-\mu_r B \\
			\end{array}
		\right].
\end{equation}
The operator ${\rm vec}$ stacks the columns of a matrix into one long vector.
Clearly, the solution space of the generalized Sylvester equation and the null space of ${\mathcal L}(\mu,\Gamma,A,B)$ have the same 
dimension. Consequently, Theorem \ref{thm:spec_eigvals_Syl} can be
rephrased as follows. 
\begin{corollary}\label{thm:spec_eigvals_rankl}
Under the assumptions of Theorem~\ref{thm:spec_eigvals_Syl}, the following two statements are equivalent.
\begin{enumerate}
	\item[\bf (1)] $\sum_{j=1}^k m_j(A,B) \geq r$.
	\item[\bf (2)] There exists $\mu \in {\mathbb S}^{r}$ such that
	$
		{\rm rank}\left(  {\mathcal L}(\mu,\Gamma,A,B)	\right)   \leq	mr - r
	$ for all $\Gamma \in {\mathcal G}(\mu)$.
\end{enumerate}
\end{corollary}

\section{A singular value characterization for the nearest pencil with specified eigenvalues}\label{sec:SVD_derivation}

As before, let ${\mathbb S} = \{ \lambda_1, \dots, \lambda_k \}$ be a set
of distinct complex scalars and let $r$ be a positive
integer. The purpose of this section is to derive a singular value optimization characterization for 
the distance $\tau_r({\mathbb S})$ defined in~(\ref{eq:taurs}). Our technique is highly inspired by the techniques in 
\cite{Mengi2009, Mengi2010} and in fact the main result of this section generalizes the singular value optimization 
characterizations from these works. We start by applying the following 
elementary result \cite[Theorem 2.5.3, p.72]{Golub1996} to the rank characterization derived in the 
previous section.
\begin{lemma} \label{thm:nearest_rankr}
Consider $C \in \C^{\ell\times q}$ and a positive integer $p < \min (\ell,q)$. Then
\[
	\inf \big\{\| \Delta C \|_2 : {\rm rank}(C+\Delta C) \leq p \big\} = \sigma_{p+1}(C).
\]
\end{lemma}

\noindent Defining
\begin{equation}\label{eq:defn_P}
	{\mathcal P}_r(\mu) :=
		\inf 
		\big\{ 
			\| \Delta A \|_2 :
				 {\rm rank} \left( {\mathcal L}(\mu, \Gamma, A+\Delta A, B) \right) \leq mr - r
		\big\}
\end{equation}
for some $\Gamma \in {\mathcal G}(\mu)$,
Corollary~\ref{thm:spec_eigvals_rankl} implies
\[
	\tau_r({\mathbb S}) = \inf_{\mu \in {\mathbb S}^r } {\mathcal P}_r(\mu),
\]
independent of the choice of $\Gamma$.
By Lemma~\ref{thm:nearest_rankr}, it holds that
\begin{eqnarray*}
	{\mathcal P}_r(\mu) 
				&= &
			\inf 
			\{ 
				\| \Delta A \|_2 : 
					 {\rm rank} \left( {\mathcal L}(\mu, \Gamma, A+\Delta A, B) \right) \leq mr - r
			\} \\
				&\geq &
		\sigma_{mr-r+1}  \left(  {\mathcal L}(\mu, \Gamma, A, B)   \right),
\end{eqnarray*}
using the fact that $A$ enters ${\mathcal L}$ linearly.
Note that this inequality in general is \emph{not} an equality
due to the fact that the allowable perturbations to ${\mathcal L}(\mu,\Gamma,A,B)$ in the definition of ${\mathcal P}_r(\mu)$  
are not arbitrary. On the other hand, the inequality holds for all $\Gamma \in {\mathcal G}(\mu)$ and hence -- by 
continuity of the singular value $\sigma_{mr-r+1}(\cdot )$ with respect to $\Gamma$ -- we obtain the lower bound
\begin{equation} \label{eq:lowerbound}
	{\mathcal P}_r(\mu) \;\;  \geq \;\; \sup_{\Gamma \in \C^{r(r-1)/2}} \sigma_{mr-r+1}  \left(  {\mathcal L}(\mu, \Gamma, A, B)   \right) =: \kappa_r(\mu).
\end{equation}
For $m = n$, it can be shown that $\sigma_{mr-r+1}  \left(  {\mathcal L}(\mu,\Gamma,A,B)  \right)$ tends to zero as
$\|\Gamma\|:=\sum |\gamma_{ij}|^2\to\infty$ provided that ${\rm rank}(B) \geq r$; see Appendix \ref{sec:to_zero}
for details. From this fact and the continuity of singular values, it
follows that the supremum is attained at some $\Gamma_{\ast}$ in the square case:
\[
	\kappa_r(\mu) = \sigma_{mr-r+1}  \left(  {\mathcal L}(\mu,\Gamma_{\ast},A,B)  \right).
\]
In the rectangular case, numerical experiments indicate that the supremum is still attained
if ${\rm rank}(B) \geq r$, but
a formal proof does not appear to be easy.
Moreover, it is not clear whether the supremum is attained at a unique $\Gamma_\ast$ or not.
However, as we will show in the subsequent two subsections, any local extremum of the singular value 
function is a global maximizer under mild assumptions. (To be precise, the satisfaction of the multiplicity 
and linear independence qualifications at a local extremum guarantees that the local 
extremum is a global maximizer;  see Definitions \ref{def:mult} and \ref{def:linin} below for multiplicity and 
linear independence qualifications.)

Throughout the rest of this section we assume that the supremum is attained at some $\Gamma_{\ast}$ and that $\Gamma_{\ast} \in
{\mathcal G}(\mu)$. The latter assumption will be removed later, in Section~\ref{sec:mainresult}.

We will establish 
the reverse inequality ${\mathcal P}_r(\mu) \leq \kappa_r(\mu)$ by constructing an optimal perturbation 
$\Delta A_{\ast}$ such that
\begin{enumerate}
	\item[\bf (i)] $\| \Delta A_{\ast} \|_2 = \kappa_r(\mu)$, $\;\;$ and
	\item[\bf (ii)] ${\rm rank} \left(  {\mathcal L}(\mu,\Gamma_{\ast},A+\Delta A_{\ast},B)  \right) \leq mr - r$.
\end{enumerate}
Let us consider the
left and right singular vectors $U\in\C^{rn}$ and $V\in\C^{rm}$ satisfying
the relations
\begin{equation}\label{eq:opt_sval_svecs}
	{\mathcal L}(\mu,\Gamma_{\ast},A,B) \; V = \kappa_r(\mu) \; U, \qquad
	U^{\ast} \; {\mathcal L}(\mu,\Gamma_{\ast},A,B) = V^{\ast} \; \kappa_r(\mu),\qquad 
\|U\|_2 = \|V\|_2 = 1.
\end{equation}
The aim of the next two subsections is to show that the perturbation
\begin{equation}\label{eq:optimal_perturbation}
	\Delta A_{\ast} := -\kappa_r(\mu)\, {\mathcal U} {\mathcal V}^+
\end{equation}
with ${\mathcal U} \in \C^{n\times r}$ and ${\mathcal V} \in \C^{m\times r}$ such that ${\rm vec}({\mathcal U}) = U$
and ${\rm vec}({\mathcal V}) = V$ satisfies properties \textbf{(i)} and \textbf{(ii)}. Here, ${\mathcal V}^{+}$ denotes 
the Moore-Penrose pseudoinverse of ${\mathcal V}$. The optimality of $\Delta A_{\ast}$
will be established under the following additional assumptions.
\begin{definition}[Multiplicity Qualification]\label{def:mult}
We say that the multiplicity qualification holds at $\left( \mu, \Gamma \right)$
for the pencil $A - \lambda B$ if the multiplicity of the singular value 
$\sigma_{mr-r+1}  \left(  {\mathcal L}(\mu,\Gamma,A,B)  \right)$ is one.
\end{definition} 
\begin{definition}[Linear Independence Qualification] \label{def:linin}
We say that the linear independence qualification holds at $\left( \mu, \Gamma \right)$
for the pencil $A - \lambda B$ if there
is a right singular vector $V$ associated with
$\sigma_{mr-r+1}  \left(  {\mathcal L}(\mu,\Gamma,A,B)  \right)$
such that $\mathcal V \in \C^{m\times r}$, with 
${\rm vec}(\mathcal V) = {V}$, has
full column rank.
\end{definition}

 \subsection{The 2-norm of the optimal perturbation} \label{sec:2norm}
Throughout this section we assume that the multiplicity qualification 
holds at the optimal $(\mu, \Gamma_{\ast})$ for the pencil $A - \lambda B$.
Moreover, we can restrict ourselves to the case $\kappa_r(\mu) \not=0$, as
the optimal perturbation is trivially given by $\Delta A_\ast = 0$
when
$\kappa_r(\mu) =0$ .

Let ${\mathcal A}(\gamma)$ be a matrix-valued function depending analytically
on a parameter $\gamma \in \R$. If the multiplicity
of $\sigma_j \left( {\mathcal A}(\gamma_\ast) \right)$ is one
and $\sigma_j \left( {\mathcal A}(\gamma_\ast) \right)\not=0$,
then $\sigma_j \left( {\mathcal A}(\gamma) \right)$ is analytic at
$\gamma = \gamma_\ast$, with the derivative
\begin{equation} \label{eq:derivative}
 	\frac{\partial \sigma_j\left( {\mathcal A}(\gamma_\ast) \right)}{\partial \gamma}
		=
	{\rm Re} \left( u_j^{\ast} \frac{\partial {\mathcal A}(\gamma_\ast)}{\partial \gamma} v_j \right),
\end{equation}
where $u_j$ and $v_j$ denote a consistent pair of unit left and right singular vectors associated
with $\sigma_j \left( {\mathcal A}(\gamma_\ast) \right)$, see, e.g.,~\cite{Bunse1991,Malyshev1999,Rel36}. 

Let us now define
\[
	f(\Gamma) := \sigma_{nr-r+1}  \big(  {\mathcal L}(\mu,\Gamma,A,B)  \big),
\]  
where we view $f$ as a mapping $\R^{r(r-1)} \rightarrow \R$ by decomposing each complex parameter $\gamma_{j\ell}$
contained in $\Gamma$ into its real and imaginary parts
$\Re \gamma_{j\ell}$ and $\Im \gamma_{j\ell}$.
By~(\ref{eq:derivative}), we have
\[
	\frac{ \partial f(\Gamma_{\ast}) }{\partial \Re \gamma_{j\ell}}
			=
	  {\rm Re} \big( U_j^{\ast} B V_\ell \big), \qquad
	  \frac{ \partial f(\Gamma_{\ast} ) }{\partial \Im \gamma_{j\ell}}
			=
	  {\rm Re} \big( \mathrm{i}\, U_j^{\ast} B V_\ell \big) = - 
{\rm Im} \big(U_j^{\ast} B V_\ell \big),
\]
where $U_j \in \C^n$ and $V_\ell \in \C^m$ denote the $j$th and $\ell$th block
components of $U$ and $V$, respectively.
Furthermore, the fact that $\Gamma_{\ast}$ is a global maximizer of $f$ implies
that both derivatives are zero. Consequently we obtain the following
result.
\begin{lemma}\label{thm:opt_svecs}
Suppose that the multiplicity qualification holds 
at $(\mu,\Gamma_{\ast})$ for the pencil $A - \lambda B$
 and $\kappa_r(\mu) \not=0$.
Then
$
	U_j^{\ast} B V_\ell = 0
$
for all $j = 2,\dots,r$ and $\ell = 1,\dots,j-1$.
\end{lemma}

Now by exploiting Lemma \ref{thm:opt_svecs} we show 
		${\mathcal U}^{\ast} {\mathcal U} = {\mathcal V}^{\ast} {\mathcal V}$.
Geometrically this means that the angle between $U_i$ and $U_j$ is identical with the
angle between $V_i$ and $V_j$.
\begin{lemma} \label{eq:uv}
Under the assumptions of Lemma \ref{thm:opt_svecs} it holds that
$
	{\mathcal U}^{\ast} {\mathcal U} = {\mathcal V}^{\ast} {\mathcal V}.
$
\end{lemma} 
\begin{proof} Expressing the first two equalities in the singular value characterization~(\ref{eq:opt_sval_svecs}) 
in matrix form yields the generalized Sylvester equations
\begin{equation*}
	A{\mathcal V}	-   B{\mathcal V}C(\mu,\Gamma_{\ast}) = \kappa_r(\mu) {\mathcal U}
\end{equation*}
and
\begin{equation*}
	{\mathcal U}^\ast A	-   C(\mu,\Gamma_{\ast}) {\mathcal U}^\ast B = \kappa_r(\mu) {\mathcal V}^\ast.
\end{equation*}
By multiplying the first equation with ${\mathcal U}^\ast$ from the left-hand side,
multiplying the second equation 
with ${\mathcal V}$ from the right-hand side, and then
subtracting the second equation from the first we 
obtain
\begin{equation}\label{eq:usu-vsv}
	\kappa_r(\mu) \left(
				{\mathcal U}^\ast {\mathcal U}		-	{\mathcal V}^\ast {\mathcal V}
				\right) 
						=
	C(\mu,\Gamma_{\ast}) {\mathcal U}^\ast B {\mathcal V}	-	{\mathcal U}^\ast B{\mathcal V} C(\mu,\Gamma_{\ast}).
\end{equation}
Lemma \ref{thm:opt_svecs} implies that ${\mathcal U}^\ast B{\mathcal V}$ is upper triangular.
Since $C(\mu,\Gamma_{\ast})$ is also upper triangular, the right-hand side in (\ref{eq:usu-vsv}) is
\emph{strictly} upper triangular. But the left-hand side in (\ref{eq:usu-vsv}) is Hermitian, implying that the right-hand
side is indeed zero, which -- together with $\kappa_r(\mu) \not=0$ -- completes the proof.
\end{proof}

The result of Lemma~\ref{eq:uv} implies $\| {\mathcal U} {\mathcal V}^{+} \|_2 = 1$.
A formal proof of this implication can be found in \cite[Lemma 2]{Malyshev1999} and \cite[Theorem 2.5]{Mengi2009}.
Indeed, the equality $\| {\mathcal U} {\mathcal V}^{+} \|_2 = 1$ can
be directly deduced from $\| {\mathcal U}{\mathcal V}^+ x\|_2 = \| {\mathcal V}{\mathcal V}^+ x\|_2$ for every $x$
(implied by Lemma~\ref{eq:uv}), and $\| VV^+ \|_2 = 1$ (since $VV^+$ is an orthogonal projector).

%
%
%
\begin{theorem}
Suppose that the multiplicity qualification holds at $(\mu,\Gamma_{\ast})$ for the pencil 
$A - \lambda B$. Then the perturbation 
$\Delta A_{\ast}$ defined in~(\ref{eq:optimal_perturbation}) satisfies
$
	\| \Delta A_{\ast} \|_2 = \kappa_r(\mu).
$
\end{theorem}

\subsection{Satisfaction of the rank condition by the optimally perturbed pencil}
Now we assume that the linear independence qualification (Definition~\ref{def:linin}) holds
at $(\mu,\Gamma_{\ast})$ for the pencil $A - \lambda B$. In particular we assume we can 
choose a right singular ``vector'' $\; {\rm vec} \left( {\mathcal V} \right) \;$ so that 
${\mathcal V}$ has full column rank. We will establish that 
\begin{equation} \label{eq:rank}
	{\rm rank} \big(  {\mathcal L} (\mu, \Gamma_{\ast},A+\Delta A_{\ast},B)  \big) \leq mr - r
\end{equation}
for $\Delta A_{\ast}$ defined as in~(\ref{eq:optimal_perturbation}).

Writing the first part of the singular vector characterization~(\ref{eq:opt_sval_svecs}) in matrix form leads
to the generalized Sylvester equation
\begin{equation*}
  A{\mathcal V}	-   B{\mathcal V}C(\mu,\Gamma_{\ast}) = \kappa_r(\mu) {\mathcal U}.
\end{equation*}
The fact that ${\mathcal V}$ has full column rank implies
${\mathcal V}^{+} {\mathcal V} = I$ and hence
\[ \displaystyle \begin{array}{rrcl}
&	A{\mathcal V} - B{\mathcal V}C(\mu,\Gamma_{\ast}) & = & \kappa_r(\mu) {\mathcal U} {\mathcal V}^{+} {\mathcal V} \\
								\Longrightarrow & 
					(A - \kappa_r(\mu) {\mathcal U} {\mathcal V}^{+}) {\mathcal V} - B{\mathcal V}C(\mu,\Gamma_{\ast}) & = & 0  \\
								\Longrightarrow & 
					(A + \Delta A_{\ast}) {\mathcal V} - B{\mathcal V}C(\mu,\Gamma_{\ast}) & = & 0. 
\end{array}
\]
Let us consider $
	{\mathcal M} =	\{
					D \in \C^{r\times r} : C(\mu,\Gamma_{\ast})D - D C(\mu,\Gamma_{\ast}) = 0
				\},
$
the subspace of all $r\times r$ matrices commuting with $C(\mu,\Gamma_{\ast})$.
By Theorem \ref{thm:Sylvester_matrix}, ${\mathcal M}$ is a subspace of dimension at least~$r$. Clearly
for all $D \in {\mathcal M}$, we have
\[
		0 =  (A + \Delta A_{\ast}) {\mathcal V}D - B{\mathcal V}C(\mu,\Gamma_{\ast})D
		   =  (A + \Delta A_{\ast}) ({\mathcal V}D) - B({\mathcal V}D) C(\mu,\Gamma_{\ast}).
\]
In other words,  
		$\{	{\mathcal V}D : D \in {\mathcal M}	\}$
has dimension at least $r$ (using the fact that ${\mathcal V}$ has full column rank)
and represents a subspace of solutions to the generalized Sylvester equation
\[
	(A + \Delta A_{\ast}) X - B X C(\mu,\Gamma_{\ast}) = 0.
\]
Reinterpreting this result in terms of the matrix representation, the desired
rank estimate~(\ref{eq:rank}) follows. This completes the derivation of 
${\mathcal P}_r(\mu) \leq \kappa_r(\mu)$ under the stated multiplicity and linear independence assumptions.



\subsection{Main Result} \label{sec:mainresult}

To summarize the discussion above, we have obtained the singular value characterization
		\begin{equation} \label{eq:svcharacterization}
		\tau_r({\mathbb S}) = 
					\inf_{\mu \in {\mathbb S}^r} \sup_{\Gamma} \sigma_{mr -r + 1} 
							\left(	{\mathcal L} \left( \mu,\Gamma,A,B \right)	\right).
		\end{equation}
Among our assumptions, we have 
\begin{equation}\label{eq:assump_removable}
	\textbf{(i)} \; \text{the KCF of $A-\lambda B$ has no right singular blocks} \;\;\;\;\;\; {\rm and} \;\;\;\;\;\; 
	\textbf{(ii)} \; \Gamma_\ast \in {\mathcal G}(\mu).
\end{equation}
In this section, we show that these two assumptions can be dropped. We still require that
${\rm rank}(B) \ge r$. As explained in the introduction, the distance problem becomes ill-posed
otherwise.

\begin{enumerate}
 \item[\bf (i)]

Suppose that the KCF of $A-\lambda B$ contains a right singular block $F_p - \lambda G_p \in \R^{p\times (p+1)}$
for some $p\ge 0$. By~\cite[Sec. 4]{Kosir96}, the generalized Sylvester equation $F_p Y - G_p X C(\mu,\Gamma) = 0$
has a solution space of dimension $r$. This implies that also the solution space of $AX - BX C(\mu,\Gamma) = 0$ has dimension
at least $r$, and consequently $\sigma_{mr -r + 1} \left( {\mathcal L} \left( \mu,\Gamma,A,B \right)\right)$ is always zero.
On the other hand, the presence of a right singular block implies that for \emph{any} $\varepsilon>0$ and $\mu_1,\ldots,\mu_r \in \C$
with $r\le \text{rank}(B)$
there is a perturbation $\triangle A$ such that $\|\triangle A\|_2\le \varepsilon$ and $(A + \triangle A) - \lambda B$
has eigenvalues $\mu_1,\ldots,\mu_r$, see~\cite{DeTK12}. 
This shows $\tau_r({\mathbb S}) = 0$ and hence both sides of~\eqref{eq:svcharacterization} are equal to zero.

In summary, we can replace the assumption \textbf{(i)} by the weaker assumption ${\rm rank}(B) \geq r$.

 \item[\bf (ii)]
To address \textbf{(ii)}, we first note that both ${\mathcal P}_r(\mu)$ and 
$\kappa_r(\mu)$, defined in  (\ref{eq:defn_P}) and (\ref{eq:lowerbound}), change continuously with 
respect to $\mu$. Suppose that $\mu$ has repeating elements, which allows for the possibilitiy
that  $\Gamma_\ast \notin {\mathcal G}(\mu)$. But for all $\tilde{\mu}$ with distinct elements, we necessarily have ${\mathcal G}(\tilde{\mu}) = \C^{r(r-1)/2}$.
Moreover, when $\tilde{\mu}$ is sufficiently close to $\mu$ then
${\mathcal P}_r(\tilde{\mu}) = \kappa_r(\tilde{\mu})$, provided that the multiplicity
 and linear independence assumptions hold at $(\mu,\Gamma_\ast)$ (implying the
 satisfaction of these two assumptions for $\tilde{\mu}$ also). Then the 
 equality ${\mathcal P}_r(\mu) = \kappa_r(\mu)$ follows from continuity. Consequently,
the assumption \textbf{(ii)} in (\ref{eq:assump_removable}) is also not needed for the 
singular value characterization.

\end{enumerate}


We conclude this section by stating the main result of this paper.
\begin{theorem}[Nearest Pencils with Specified Eigenvalues]\label{thm:sval_char_pres_eigs}
Let $A - \lambda B$ be an $n\times m$ pencil with $n \geq m$, let $r$
be a positive integer such that $r \leq {\rm rank}(B)$ and let ${\mathbb S} = \{ \lambda_1, \dots, \lambda_k \}$ 
be a set of distinct complex scalars.
\begin{enumerate}
	\item[\bf (i)] Then
		\[
		\tau_r({\mathbb S}) = 
					\inf_{\mu \in {\mathbb S}^r} \sup_{\Gamma} \sigma_{mr -r + 1} 
							\left(	{\mathcal L} \left( \mu,\Gamma,A,B \right)	\right)
		\]
	holds, provided that the optimization problem on the right is attained at some $(\mu_{\ast}, \Gamma_{\ast})$ for which 
$\Gamma_\ast$ is finite and the 
	multiplicity as well as the linear independence qualifications hold.
	\item[\bf (ii)] A minimal perturbation $\Delta A_{\ast}$ such that $\sum_{j=1}^k m(A + \Delta A_{\ast}, B) \geq r$
	is given by (\ref{eq:optimal_perturbation}), with $\mu$ replaced by $\mu_{\ast}$.
\end{enumerate}
\end{theorem}

\section{Corollaries of Theorem \ref{thm:sval_char_pres_eigs}}\label{sec:nearest_rect_pencil}
As discussed in the introduction one potential application of Theorem~\ref{thm:sval_char_pres_eigs}
is in control theory, to ensure 
that the eigenvalues lie in a particular region in the complex plane. Thus let $\Omega$ be a subset
of the complex plane. Then, provided that the assumptions
of Theorem~\ref{thm:sval_char_pres_eigs} hold, we have the following singular value characterization for the distance to the
nearest pencil with $r$ eigenvalues in $\Omega$:
\begin{equation}\label{eq:nearest_pencil_region}
\begin{split}
	\tau_r(\Omega) := & \inf_{{\mathbb S} \subseteq \Omega} \tau_r({\mathbb S}) \hskip 32ex  \\
				 = & \inf_{{\mathbb S} \subseteq \Omega} \inf_{\mu \in {\mathbb S}^r} \sup_{\Gamma} \sigma_{mr -r + 1} 
							\big(	{\mathcal L} \left( \mu,\Gamma,A,B \right)	\big) \\
				= & \inf_{\mu \in \Omega^r} \sup_{\Gamma} \sigma_{mr -r + 1} 
							\big(	{\mathcal L} \left( \mu,\Gamma,A,B \right)	\big), \hskip 7ex
\end{split}
\end{equation}
where $\Omega^r$ denotes the set of vectors of length $r$ with all entries in $\Omega$. 

When the pencil $A - \lambda B$ is rectangular, that is $n > m$, the pencil 
has generically no eigenvalues. Then the distance to the nearest rectangular pencil with $r$ eigenvalues is of
interest. In this case, the singular value characterization takes the following form:
\begin{equation}\label{eq:nearest_pencil_cplane}
	\tau_r(\C) =  \inf_{\mu \in \C^r} \sup_{\Gamma} \sigma_{mr -r + 1} 
							\left(	{\mathcal L} \left( \mu,\Gamma,A,B \right)	\right).
\end{equation}
The optimal perturbations $\Delta A_{\ast}$ such that the pencil $(A + \Delta A_{\ast}) - \lambda B$ 
has eigenvalues (in $\C$ and $\Omega$) are given by (\ref{eq:optimal_perturbation}),
with $\mu$ replaced by the minimizing $\mu$ values in (\ref{eq:nearest_pencil_cplane}) 
and (\ref{eq:nearest_pencil_region}), respectively.


\section{Multiplicity and linear independence qualifications}\label{sec:qualifications}

The results in this paper are proved under the assumptions of
multiplicity and linear independence qualifications. This section provides 
an example for which the multiplicity and
linear independence qualifications are not satisfied for the optimal value of
$\Gamma$. Note that this does not mean that these assumptions are necessary to
prove the results from this paper. In fact, numerical experiments suggest
that our results may hold even if these assumptions are
not satisfied.

Consider the pencil
\[ 
\left[
\begin{array}{rcc}
-1 & 0 & 0 \\
0 & 5 & 0 \\
0 & 0 & 2
\end{array}
\right]
-\lambda
\begin{bmatrix}
0 & 0 & 0 \\
0 & 1 & 0 \\
0 & 0 & 1 
 \end{bmatrix}. \]
Let $\mu=\begin{bmatrix} 5 & 1 \end{bmatrix}^T$, that
is, the target eigenvalues are $5$ and $1$. 
Then it is easy to see that the optimal perturbation is given by 
\[ \Delta A_{\ast} = \left[ \begin{array}{ccr}   0 & 0 & 0 \\ 0 & 0 & 0 \\ 0 & 0 & -1 \end{array} \right]. \]
The singular values of the matrix ${\mathcal L}(\mu,\gamma,A,B)$ are
\[ 0, 1, \sqrt{ 16+|\gamma|^2 },
\sqrt{ 5+\frac{1}{2}|\gamma|^2\pm \frac{1}{2} \sqrt{|\gamma|^4+20|\gamma|^2+64} } \]
where the multiplicity of the singular value 1 is two. Hence 
\[
	\sigma_5\left({\mathcal L}(\mu,\gamma,A,B) \right) = 
		\sqrt{
			5+\frac{1}{2}|\gamma|^2 - \frac{1}{2} \sqrt{|\gamma|^4+20|\gamma|^2+64}
			}.
\] 
Clearly the supremum is attained for $\gamma=0$ and $\sigma_5\left({\mathcal L}(\mu,0,A,B) \right) = 1$. 
Hence the multiplicity condition at the optimal $\gamma$ is violated. All three pairs of singular vectors corresponding
to the singular value 1 at the optimal $\gamma$ violate the linear independence condition, but one pair does 
lead to the optimal perturbation $\Delta A_{\ast}$.

\section{Computational issues}\label{sec:computation}

A numerical technique that can be used to compute $\tau_r(\Omega)$ and $\tau_r(\C)$
based on the singular value characterizations was already described in \cite{Mengi2009, Mengi2010}.
For completeness, we briefly recall this technique in the following. The
distances of interest can be characterized as
\[
	\tau_r(\Omega) = \inf_{\mu \in \Omega^r} g(\mu) \hskip 7ex 
					{\rm and}  \hskip 7ex
	\tau_r(\C) = \inf_{\mu \in \C^r} g(\mu),
\]
where $g : \C^r \rightarrow \R$ is defined by
\[
	g(\mu) := \sup_{\Gamma \in \C^{r(r-1)/2}} \sigma_{mr-r+1}  \big(  {\mathcal L}(\mu,\Gamma,A,B)  \big).
\]

The inner maximization problems are solved by BFGS, even though $\sigma_{mr-r+1}(\cdot)$ 
is not differentiable at multiple singular values. In practice this is not a major issue for BFGS as long as
a proper line search (e.g., a line search respecting weak Wolfe conditions) is used, as the multiplicity 
of the $r$th smallest singular value is one generically with respect to $\Gamma$  for any given $\mu$;
see the discussions in~\cite{LO12}. If the multiplicity and linear independence 
qualifications hold at a local maximizer $\Gamma_{\ast}$, then $\Gamma_{\ast}$ is in fact a global 
maximizer and hence $g(\mu)$ is retrieved. If, on the other hand, BFGS converges to a point where 
one of these qualifications is violated, it needs to be restarted with a different initial guess. In practice 
we have almost always observed convergence to a global maximizer immediately,
without the need for such a restart.

Although the function $g(\mu)$ is in general non-convex, it is Lipschitz continuous:
\[
	| g(\mu + \delta \mu) - g(\mu) | \leq \| \delta\mu \|_2 \cdot \| B \|_2.
\]
There are various Lipschitz-based global optimization algorithms in the literature stemming 
mainly from ideas due to Piyavskii  and Shubert (see \cite{Piyavskii1972, Shubert1972}). The Piyavskii-Shubert
algorithm is based on the idea of constructing a piecewise linear approximation lying beneath the Lipschitz 
function. We used DIRECT (see \cite{Jones1993}), a sophisticated variant of the Piyavskii-Shubert algorithm.
DIRECT attempts to estimate the Lipschitz constant locally, which can possibly speed up 
convergence.

The main computational cost involved in the numerical optimization of singular values is 
the retrieval of the $r$th smallest singular value of ${\mathcal L}(\mu,\Gamma,A,B)$ 
at various values of $\mu$ and $\Gamma$. As we only experimented with small pencils,
we used direct solvers for this purpose. For medium to large scale
pencils, iterative algorithms such as the Lanczos method (see \cite{Golub1996}) are more appropriate.

\section{Numerical Experiments}\label{sec:numerical_exp}
Our algorithm is implemented in Fortran, calling routines
from LAPACK for singular value computations, the limited memory BFGS routine
written by J. Nocedal (discussed in \cite{Liu1989}) for inner maximization problems, 
and an implementation of the DIRECT algorithm by Gablonsky (described 
in \cite{Gablonsky2001}) for outer Lipschitz-based minimization. A mex interface provides 
convenient access via {\sc Matlab}.

The current implementation is not very reliable, which appears to be
related to the numerical solution of the outer Lipschitz minimization 
problem, in particular the DIRECT algorithm and its termination criteria.
We rarely obtain results that are less accurate than the prescribed
accuracy. The multiplicity and linear independence qualifications usually 
hold in practice and don't appear to affect the numerical accuracy.
For the moment, the implementation is intended for small
pencils (\textit{e.g.}, $n,m < 100$). 

\subsection{Nearest Pencils with Multiple Eigenvalues}
As a corollary of Theorem \ref{thm:sval_char_pres_eigs} it follows that 
for a square pencil $A - \lambda B$ the nearest pencil having 
${\mathbb S} = \{ \mu \}$ as a multiple eigenvalue is given by
\[
	\tau_2({\mathbb S}) =
	\sup_{\gamma}
	\left(
		\left[
			\begin{array}{cc}
						  A - \mu B & 0 \\
					\gamma B		& A - \mu B \\
			\end{array}
		\right]
	\right)
\]
provided that the multiplicity and linear independence qualifications are satisfied at the
optimal $(\mu,\gamma_\ast)$. Therefore, for the distance from $A - \lambda B$ to the nearest 
square pencil with a multiple eigenvalue the singular value characterization takes the
form
\begin{equation}\label{eq:dist_defect}
	\inf_{\mu \in \C}
	\sup_{\gamma}
	\sigma_{2n-1}
	\left(
		\left[
			\begin{array}{cc}
				A - \mu B & 0 \\
				\gamma B		& A - \mu B \\
			\end{array}
		\right]
	\right).
\end{equation}

Specifically, we consider the pencil
\begin{equation}\label{eq:num_examp1_pencil}
	A-\lambda B = 
	\left[
		\begin{array}{rrr}
			 2 & -1 & -1 \\
			-1 &  2 & -1 \\
			-1 & -1 &  2 \\ 
		\end{array}
	\right]
			-
		\lambda
	\left[
		\begin{array}{rrr}
			-1 & 2 & 3 \\
			 2 & -1 & 2 \\
			 4 &  2  & -1 \\
		\end{array}
	\right].
\end{equation}
Solving the above singular value optimization problem results in a distance of $0.59299$ to the nearest
pencil with a multiple eigenvalue. By~(\ref{eq:optimal_perturbation}), a nearest pencil turns out to be
\[
	\left[
		\begin{array}{rrr}
			1.91465 & -0.57896 & -1.21173 \\
		        -1.32160 &  1.93256 &  -0.57897 \\
		        -0.72082 & -1.32160 &  1.91466 \\
		\end{array}
	\right]
				-
			\lambda
	\left[
		\begin{array}{rrr}
			-1 & 2 & 3 \\
			 2 & -1 & 2 \\
			 4 &  2  & -1 \\
		\end{array}
	\right],
\]
with the double eigenvalue $\lambda_{\ast} = -0.85488$. The optimal maximizing $\gamma$ turns out
to be zero, which means neither the multiplicity nor the linear independence qualifications hold.
(This is the non-generic case; had we attempted to calculate the distance to the nearest pencil
with $\mu$ as a multiple eigenvalue for a given $\mu$, optimal $\gamma$ appears to be non-zero 
for generic values of $\mu$.)  Nevertheless, the singular value characterization (\ref{eq:dist_defect}) 
remains to be true for the distance as discussed next.  

The $\epsilon$-pseudospectrum of $A - \lambda B$ (subject to perturbations in $A$ only) is the set 
$\Lambda_{\epsilon}(A,B)$ containing the eigenvalues of all pencils $(A+\Delta A) - \lambda B$ such 
that $\| \Delta A \|_2 \leq \epsilon$.  Equivalently,
\[
	\Lambda_{\epsilon}(A,B) = \{
							\lambda \in \C :
							\sigma_{\min}(A - \lambda B) \leq \epsilon
						  \}.
\]
It is well known that the smallest $\epsilon$ such that two components of $\Lambda_{\epsilon}(A,B)$
coalesce equals the distance to the nearest pencil with multiple eigenvalues. (See~\cite{Alam2005}
for the case $B = I$, but the result easily extends to arbitrary invertible $B$.)
Figure \ref{fig:gen_pseudo} displays the pseudospectra of the pencil in~(\ref{eq:num_examp1_pencil}) for 
various levels of $\epsilon$.
Indeed, two components of the $\epsilon$-pseudospectrum coalesce for $\epsilon = 0.59299$,
confirming our result.  
\begin{figure}
   	\begin{center}
			\includegraphics[height=0.4\vsize,]{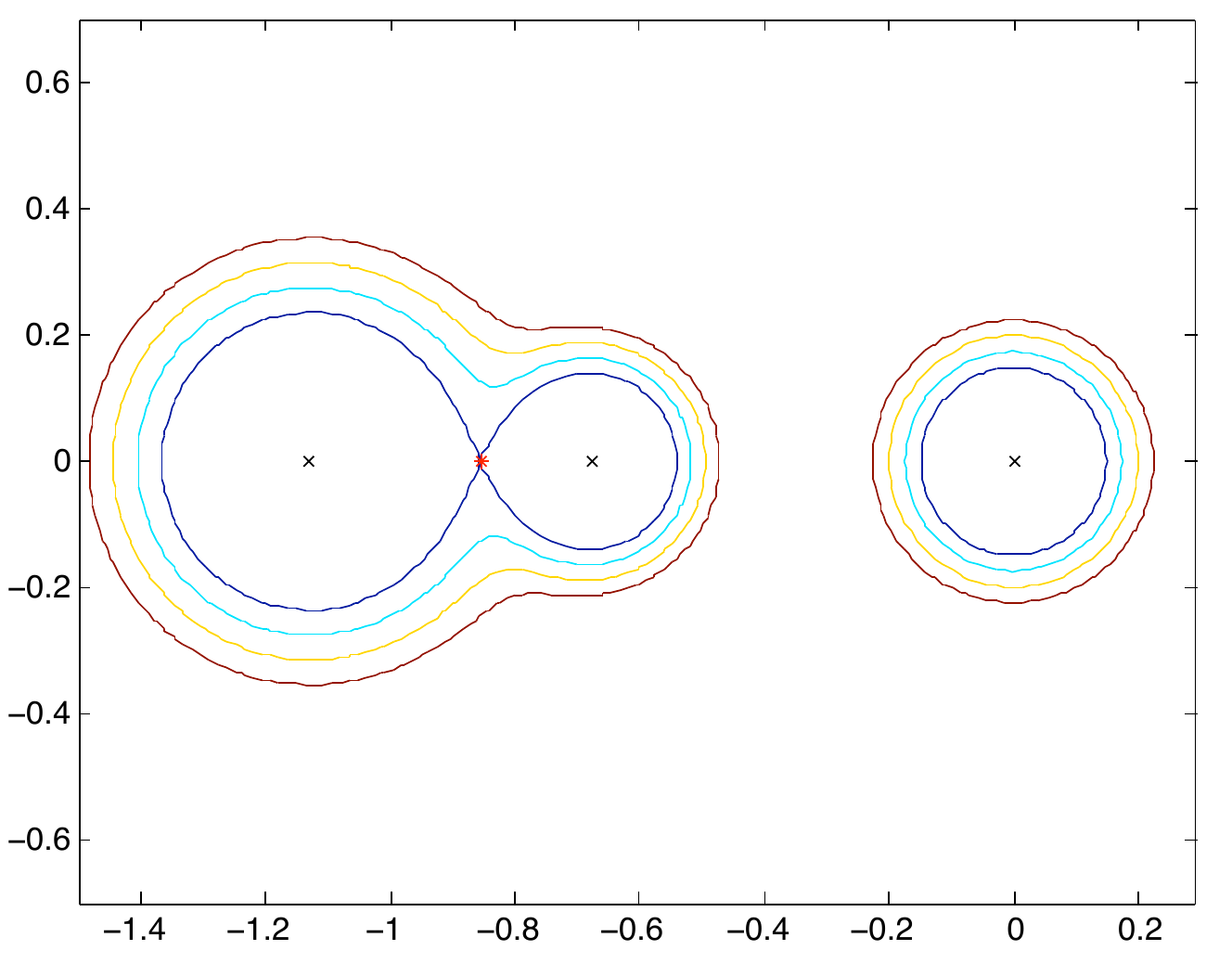}
	\end{center}
	    \caption{Pseudospectra for the pencil in (\ref{eq:num_examp1_pencil}),
with eigenvalues marked by the black crosses.
Two components of the $\epsilon$-pseudospectrum coalesce for $\epsilon = 0.59299$, corresponding
to the distance to a nearest pencil with a multiple eigenvalue $\lambda_\ast = -0.85488$ at the coalescence 
point (marked by the asterisk).}\label{fig:gen_pseudo}
\end{figure}

\subsection{Nearest Rectangular Pencils with at least Two Eigenvalues}

As an example for a rectangular pencil, let us consider the $4\times 3$ pencil
\[
	A - \lambda B = 
	\left[
		\begin{array}{rrr}
		   	1 & 0 	& 0 \\
			0 & 0.1	& 0 \\
			0 & 2		& 0.3 \\
			0 & 1		& 2 \\
		\end{array}
	\right]
			-
	\lambda
	\left[
		\begin{array}{ccc}
			0 & 0 & 0 \\
			1 & 0 & 0 \\
			0 & 1 & 0 \\
			0 & 0 & 1 \\
		\end{array}
	\right].
\]
The KCF of this pencil contains a $4\times 3$ singular block and therefore the pencil has no eigenvalues. 
However, if the entry ${a}_{22}$ is set to zero, the
KCF of the resulting pencil contains a
$2 \times 1$ singular block and a $2\times 2$ regular block 
corresponding to finite eigenvalues. Hence, a perturbation with 2-norm
$0.1$ is sufficient to have two eigenvalues.

According to the corollaries in Section \ref{sec:nearest_rect_pencil} the distance to the 
nearest $4\times 3$ pencil with at least two eigenvalues has the characterization
\begin{equation}\label{eq:rect2_svalchar}
	\tau_2(\C)
			=
	\inf_{\mu \in \C^2}
	\underbrace{
	\sup_{\gamma}
	\sigma_{2m-1}
	\left(
	\left[
		\begin{array}{cc}
			{A} - \mu_1 {B} & 0 \\
			    	\gamma {B} 	  & {A} - \mu_2 {B} \\
		\end{array}
	\right]
	\right)}_{=:g(\mu)}
\end{equation}
for $m = 3$. Our implementation returns $\tau_2(\C) = 0.03927$. The corresponding nearest 
pencil~(\ref{eq:optimal_perturbation}) is given by
\[
	\left[
		\begin{array}{rrr}
			0.99847 & -0.03697 & -0.01283 \\
				0    &	  0.08698 &  0.03689 \\
				0    &  2.00172 &  0.30078 \\
			0.00007 & 1.00095  &  2.00376 \\
		\end{array}	
	\right]	
			-
	\lambda
	\left[
		\begin{array}{ccc}
			0 & 0 & 0 \\
			1 & 0 & 0 \\
			0 & 1 & 0 \\
			0 & 0 & 1 \\
		\end{array}
	\right]
\]
and has eigenvalues at $\mu_1 = 2.55144$ and $\mu_2 = 1.45405$. This result is confirmed
by Figure~\ref{fig:Jordan_rect}, which illustrates the level
sets of the function $g(\mu)$ defined in (\ref{eq:rect2_svalchar}) over $\R^2$.

For this example the optimal $\gamma$ is 2.0086. The smallest three singular values
of the matrix in (\ref{eq:rect2_svalchar}) are $1.4832$, $0.0393$ and $0.0062$ for these optimal
values of $\mu$ and $\gamma$. The linear independence qualification also holds.

\begin{figure}
   	\begin{center}
			\includegraphics[height=0.55\vsize,]{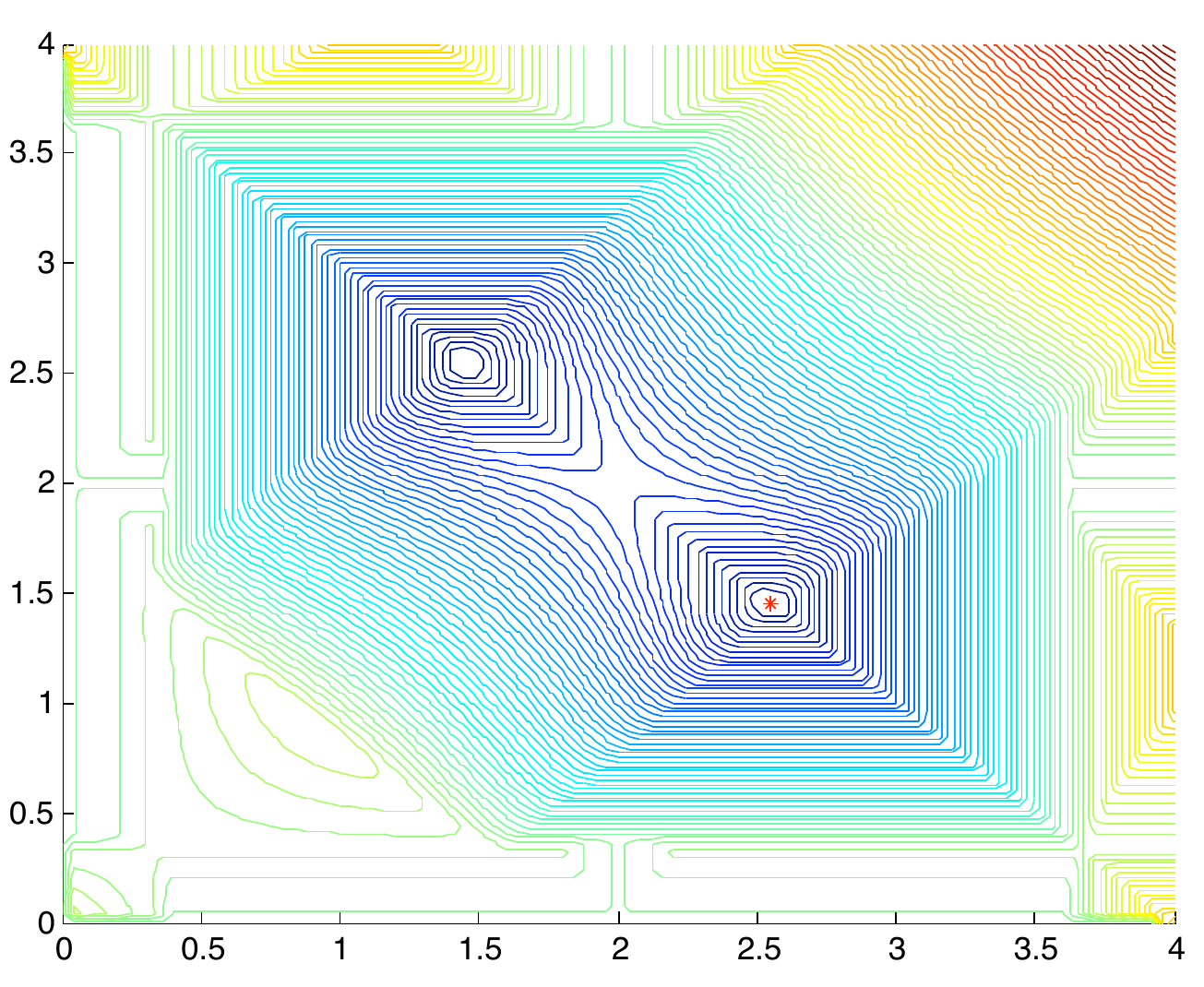}
	\end{center}
	    \caption{Level sets over $\mathbb R^2$ of the function $g(\mu)$ defined
	    in (\ref{eq:rect2_svalchar}). The asterisk marks
	    the numerically computed global minimizer of $g$, which corresponds to the eigenvalues
	    of a nearest pencil with two eigenvalues.}\label{fig:Jordan_rect}
\end{figure}

\subsection{Nearest Stable Pencils}
As a last example, suppose that $Bx^\prime(t) = Ax(t)$ with $A, B \in \C^{n\times n}$ is an unstable descriptor system. 
The distance to a nearest stable descriptor system is a special case of $\tau_n(\Omega)$,
with $\Omega = \C^{-}$, the open left-half of the complex plane. A singular 
value characterization is given by
\[
	\tau_n(\C^{-})
			=
	\inf_{\lambda_j \in \C^{-}} \;\; \sup_{\gamma_{ik} \in \C} \;
	\sigma_{n^2 - n + 1}
	\left(
		\left[
			\begin{array}{cccc}
				 A - \lambda_1 B & 0 & & 0 \\
				 \gamma_{21} B & A - \lambda_2 B & & 0 \\
				       & & \ddots & \\
				 \gamma_{n1} B & \gamma_{n2} B & & A - \lambda_n B \\
			\end{array}
	\right]
	\right).
\]

Specifically, we consider a system with $B = I_2$ and
\begin{equation}\label{eq:2by2unstable}
	A =
		\left[
			\begin{array}{cc}
				0.6 - \frac{1}{3} i & -0.2 + \frac{4}{3} i \\
			        -0.1 + \frac{2}{3} i  &  0.5 + \frac{1}{3} i  \\
			\end{array}
		\right]. 
\end{equation}
Both eigenvalues $\lambda_1 = 0.7 - i$ and $\lambda_2 = 0.4 + i$ are in the
right-half plane. Based on the singular value characterization, we have computed the distance to a nearest stable
system $x'(t) = (A + \Delta A_{\ast}) x(t)$ 
as $0.6610$. The corresponding perturbed matrix
\[
	A + \Delta A_{\ast}
			=
	\left[
			\begin{array}{rr}
				0.0681 - 0.3064i & -0.4629 + 1.2524i \\
			         0.2047 + 0.5858i &  -0.1573 + 0.3064i \\
			\end{array}
	\right] 
\]
at a distance of $0.6610$ has one eigenvalue $\left( \lambda_{\ast} \right)_1 = -0.0885+0.9547i$ 
in the left-half plane and the other  $\left( \lambda_{\ast} \right)_2 = -0.9547i$ on the imaginary axis. 
The $\epsilon$-pseudospectrum of $A$ is depicted in Figure \ref{fig:pss_nearest_stable}. For
$\epsilon = 0.6610$, one component of the 
$\epsilon$-pseudospectrum crosses the imaginary axis, while the other component touches the 
imaginary axis.

\begin{figure}
   	\begin{center}
			\includegraphics[height=0.45\vsize,]{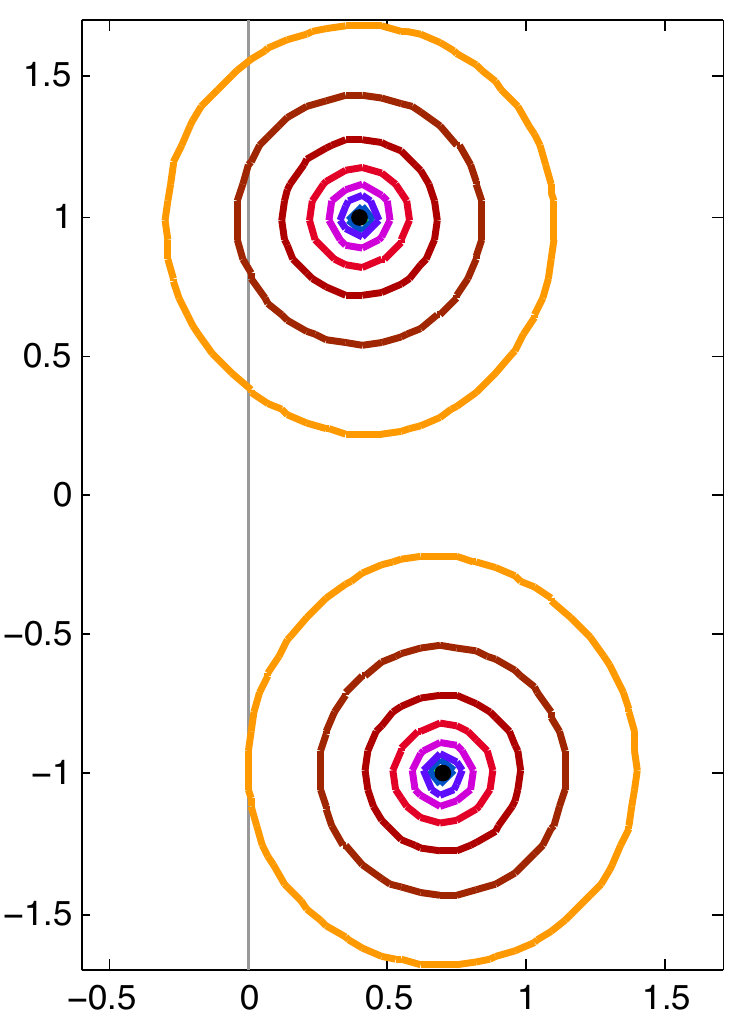}
	\end{center}
	    \caption{Pseudospectra of the matrix $A$ defined
	    in (\ref{eq:2by2unstable}). The outer orange curve represents 
	    the boundary of the $\epsilon$-pseudospectrum for $\epsilon = 0.6610$, the 
	    distance to a nearest stable matrix. }\label{fig:pss_nearest_stable}
\end{figure}

\section{Concluding Remarks}
In this work a singular value characterization has been derived for the 2-norm of a smallest perturbation 
to a square or a rectangular pencil $A - \lambda B$ such that the perturbed pencil has a desired 
set of eigenvalues. The immediate corollaries of this main result are
\begin{enumerate}
	\item[\bf (i)] a singular value characterization for the 2-norm of the smallest perturbation 
	so that the perturbed pencil has a specified number of its eigenvalues in a desired region 
	in the complex plane, and
	\item[\bf (ii)] a singular value characterization for the 2-norm of the smallest perturbation
	to a rectangular pencil so that it has a specified number of eigenvalues.
\end{enumerate}

Partly motivated by an application explained in the introduction, we allow perturbations to $A$ only.
The extension of our results to the case of simultaneously perturbed $A$ and $B$ remains open.

The development of efficient and reliable computational techniques for the solution of the derived 
singular value optimization problems is still in progress. As of now the optimization problems can
be solved numerically only for small pencils with small number of desired eigenvalues. The 
main task that needs to be addressed from a computational point of view is a reliable and 
efficient implementation of the DIRECT algorithm for Lipschitz-based optimization. For large
pencils it is necessary to develop Lipschitz-based algorithms converging asymptotically 
faster than the algorithms (such as the DIRECT algorithm) stemming from the Piyavskii-Shubert 
algorithm. The derivatives from Section~\ref{sec:2norm} might constitute a first step in this
direction.

\vskip 2ex

\noindent
\textbf{Acknowledgments}
We are grateful to two anonymous referees for their valuable comments. The research of 
the second author is supported in part by the European Commision grant
PIRG-GA-268355 and the T\"{U}B\.{I}TAK (the scientific and technological research 
council of Turkey) carrier grant 109T660.

\vskip 6ex

\appendix
\section{ Proof that $\sigma_{mr-r+1}  \left(  {\mathcal L}(\mu,\Gamma,A,B)  \right) \to 0$ as
$\Gamma\to\infty$} \label{sec:to_zero}

We prove that the $r$ smallest singular values of ${\mathcal L}(\mu,\Gamma,A,B)$
decay to zero as soon as at least one entry of $\Gamma$ tends to infinity, provided that $n = m$.
In the rectangular case, $n>m$, these singular values generally do not decay to zero.

We start by additionally assuming that
$A-\mu_i B$ are non--singular matrices for all $i=1,\ldots,r$.
We will first prove the result under this assumption, and then we will drop it.
Our approach is a generalization of the procedure from \cite[\S 5]{IN2005}, which in turn is a generalization 
of \cite[Lemma 2]{Malyshev1999}.

Under our assumptions the matrix ${\mathcal L}(\mu,\Gamma,A,B)$ is non--singular, and one can explicitly 
calculate the inverse. It is easy to see that the matrix ${\mathcal L}^{-1}(\mu,\Gamma,A,B)$ has the form 
\[\begin{bmatrix}
(A-\mu_1 B)^{-1} & 0 & \ldots & 0 \\
X_{21} & (A-\mu_2 B)^{-1} & \ldots & 0 \\
X_{31} & X_{32} & \ldots & 0 \\
\vdots & \vdots & \ddots & \vdots \\
X_{r1} & X_{r2} & \ldots & (A-\mu_r B)^{-1}
\end{bmatrix}.
\]
We will use the well--known relations
\begin{equation}
\label{eq:sigmar}
\sigma_{nr-r+1}  \left(  {\mathcal L}(\mu,\Gamma,A,B)  \right) = \sigma_{r}\left({\mathcal L}(\mu,\Gamma,A,B)^{-1} \right)^{-1} \le \sigma_{r}\left(X_{ij} \right)^{-1}. 	
\end{equation}
We first compute the matrices $X_{21},\ldots,X_{r,r-1}$ which lie on the first sub--diagonal. By a straightforward computation we obtain
\begin{equation*}
	X_{i+1,i} = -\gamma_{i+1,i}(A-\mu_{i+1}B)^{-1}B(A-\mu_i B)^{-1}.
\end{equation*}
If $\sigma_r \left( (A-\mu_{i+1}B)^{-1}B(A-\mu_i B)^{-1}\right) >0 $, then from \eqref{eq:sigmar} it follows that if any of $|\gamma_{i+1,i}|$ 
tends to infinity, we obtain the desired result. But $\sigma_r \left( (A-\mu_{i+1}B)^{-1}B(A-\mu_i B)^{-1}\right) > 0 $ easily 
follows from the assumption ${\rm rank} (B) \ge r$.

If this is not the case, meaning $\max_i\{\gamma_{i+1,i}\}$ is bounded, then 
we use the entries on the next sub--diagonal $X_{i+2,i}$. Again by straightforward computation we obtain
\[ X_{i+2,i} = -\gamma_{i+2,i}(A-\mu_{i+2}B)^{-1}B(A-\mu_i B)^{-1} + 
\gamma_{i+2,i+1}\gamma_{i+1,i}(A-\mu_{i+2}B)^{-1}B(A-\mu_{i+1}B)^{-1}B(A-\mu_i B)^{-1}. \]
Because again ${\rm rank} (B) \ge r$ implies  $\sigma_r \left( (A-\mu_{i+2}B)^{-1}B(A-\mu_i B)^{-1} \right) >0 $, it follows that if any of $|\gamma_{i+2,i}|$ tend to infinity, we obtain the desired result.
In general, we have the recursive formula
\[ X_{i+j,i} = -\gamma_{i+j,i}(A-\mu_{i+j}B)^{-1}B(A-\mu_i B)^{-1} - 
\sum_{k=1}^{j-1} \gamma_{i+j,i+k} (A-\mu_{i+j}B)^{-1}BX_{i+k,i}. \]
Applying the same procedure as above, we conclude the proof in this case.

To remove the assumption that the matrices $A-\mu_i B$ are non--singular, we fix any $\varepsilon >0$. Let us choose a matrix $A_{\varepsilon}$ such that 
$\|A_{\varepsilon}-A\| < \varepsilon$ and that the matrices $A_{\varepsilon}-\mu_i B$ are non--singular for all $i=1,\ldots,r$. 
From the arguments above, if follows that there exists $\gamma_0>0$ such that 
$ \sigma_{nr-r+1}  \left(  {\mathcal L}(\mu,\Gamma,A_{\varepsilon},B)  \right) < \varepsilon$, when $\|\Gamma\|>\gamma_0$. Since 
\begin{center}
	$\sigma_{nr-r+1}  \left(  {\mathcal L}(\mu,\Gamma,A,B)  \right) \le \sigma_{nr-r+1}  \left(  {\mathcal L}(\mu,\Gamma,A_{\varepsilon},B)  \right) + \varepsilon$, 
\end{center}
we obtain the inequality $\sigma_{nr-r+1}  \left(  {\mathcal L}(\mu,\Gamma,A,B)  \right) < 2 \varepsilon $, when $\|\Gamma\|>\gamma_0$.

\bibliography{GEP2}
\end{document}